\documentclass[10pt,oneside]{amsart}
\usepackage{amssymb, amscd, amsmath, amsthm, latexsym, hyperref,amsrefs}
\usepackage{pb-diagram, fancyhdr, graphicx, psfrag, enumerate}
\usepackage[text={6.2in,8.5in},centering,letterpaper,dvips]{geometry}
\usepackage{enumitem}\setlist[enumerate,1]{font=\upshape, itemsep=.5ex}\setlist[itemize,1]{font=\upshape, itemsep=.5ex}
\usepackage{pinlabel}
\usepackage{caption}
\usepackage{xcolor}

\def\Z{{\mathbb Z}}

\def\R{{\mathbb R}}

\def\S{{\mathbb S}}
\def\A{{\mathbb A}}

\def\call{\mathcal{L}}

\def\calm{\mathcal{M}}

\def\diff{\mathcal{D}}

\def\SO{\textrm{SO}}

\def\Sn{{\mathbb S}_n}

\def\diff{ { \sf {Diff} }}

\def\co{\colon\thinspace}

\newtheorem{theorem}{Theorem} [section]
\newtheorem{lemma}[theorem]{Lemma}
\newtheorem{corollary}[theorem]{Corollary}

\theoremstyle{definition}

\newtheorem{definition}[theorem]{Definition}
\newtheorem{example}[theorem]{Example}

\begin{document}
\title{intrinsic symmetry groups of links}

\author{Charles Livingston}
\thanks{This work was supported by a grant from the National Science Foundation, NSF-DMS-1505586.   }
\address{Charles Livingston: Department of Mathematics, Indiana University, Bloomington, IN 47405}\email{livingst@indiana.edu}



\begin{abstract}   

The set of isotopy classes of ordered $n$--component  links in $S^3$ is acted on by the symmetric group, $\Sn$, via permutation of the components.  The   subgroup  $\S(L) \subset \Sn$ is defined to be  the set of  elements in the symmetric group that preserve the ordered  isotopy type of $L$ as an unoriented link.   The study of these groups  was initiated in 1969, but the question of whether or not every subgroup of $\Sn$ arises as such an intrinsic  symmetry group of some link has remained open.  We provide counterexamples; in particular,  if  $n \ge 6$ , then there does not exist an  $n$--component link $L$ for which $\S(L)$ is  the alternating group $\A_n$. 


\end{abstract}

\maketitle


\section{Introduction} \label{sec:introduction} 


The {\it oriented diffeomorphism group}  of an ordered link $L = \{L_1, \ldots , L_n\} \subset S^3$ consists of  all orientation preserving diffeomorphisms of $S^3$ that preserve the link setwise.  We denote this group  $\diff(L)$.   The action of   $\diff(L)$ on the components of $L$ defines a homomorphism from $\diff(L)$  to    the symmetric group $\S_n$; its image  is denoted $\S(L)$.   A basic question asks whether  every subgroup $H \subset \S_n$ arises as $\S(L)$ for some $n$--component link.  We provide obstructions.  Our   examples of groups that do not arise are the  alternating groups,  $\A_n$, for $n\ge 6$.   \smallskip

\noindent{\bf Theorem.} {\it If $n \ge 6$, then there does not exist an ordered $n$--component link  $L$ that satisfies $\S(L) = \A_n$.}

\smallskip

The study of symmetries of links is usually placed in the context of an extension of the symmetric group called the {\it Whitten group}:  
\[
\Gamma_n = \Z_2 \oplus \big(     (\Z_2)^n \rtimes \Sn \big).
\]
In the semidirect product, $\Sn$ acts on $(\Z_2)^n$ in by permuting the coordinates.  As we will describe in the next section,  the $\Z_2$ factors keep track of the orientations of $S^3$  and   the components of $L$.  The question of which subgroups of $\Gamma_n$ arise from links was first  considered by  Fox and    Whitten in the mid-1960s, first appearing in print in 1969~\cite{MR242146}.  The theorem stated above provides the first examples of groups that cannot arise.
\smallskip

\noindent{\bf Summary of proof.} The basic idea of our approach is as follows.  For a given link $L$ there is a Jaco-Shalen-Johannson   JSJ--decomposition of the complement of $L$ into hyperbolic and Seifert fibered components $\{C_i\}$.  This decomposition is unique up to isotopy.  We first observe that if   $\S(L)$ does not contain an index two subgroup, then   one of the $C_i$ (say $C_1$) is invariant under the action of $\diff(L)$ up to isotopy.

If $C_1$ is hyperbolic, we can replace the action of $\diff(L)$ restricted to $C_1$ with a finite group of isometries of $C_1$.  We   then use a  reembedding of $C_1$ into $S^3$ (as first  described  by Budney in~\cite{MR2300613}) to  extend that action to $S^3$.  It   follows from results such as~\cite{MR2178962} that the action on $S^3$ is conjugate to a  linear action.  We then find that $\S(L)$ is a quotient of a finite subgroup of $\SO(4)$.   Finally,  a group theoretic analysis reduces the problem to the simpler one of considering quotients of finite subgroups of $\SO(3)$, which  are enumerated.

In contrast to the hyperbolic case,  if   $C_1$ is Seifert fibered, then the diffeomorphism group of $C_1$ itself  is large, sufficiently so that we can construct enough symmetries of $L$ to show that  $\S(L) = \Sn$.  

\smallskip

\noindent{\bf Outline.}  Section~\ref{sec:generalities} describes the general theory of intrinsic symmetry groups of  oriented links, as first considered by Fox and Whitten~\cite{MR242146}.  Sections~\ref{sec:knots} and~\ref{sec:links} describe the classical case of knots, $n=1$, and results for the case of $n=2$.
Section~\ref{sec:fully}  presents prime, non-split links, with full symmetry group for all $n$.

In Section~\ref{sec:jsw-trees} we describe JSJ-decompositions, the associated tree diagrams, and prove that in the case of $\S(L) = \A_n$, some component of the decomposition is fixed (up to isotopy) by the action of the diffeomorphism group.  
Section~\ref{sec:reembed} explains how that distinguished component can be reembedded into $S^3$ as the complement of a link.  The reembedding is used in   Section~\ref{sec:hyperbolic} to show that if the fixed component is hyperbolic, then $\S(L)$ is a subgroup of a quotient of a  finite subgroup of $SO(4)$.  Finally, in Section~\ref{sec:seifertfibered} we present the Seifert fibered case.  In the concluding Section~\ref{sec:questions}, we present a few questions and include an example of a 4--component link $L$ with $\S(L) = \A_4$.
\smallskip

\noindent{\bf Notational comment.}  We are calling the groups studied here the {\it intrinsic symmetry groups} of links.  The {\it symmetry group} of a link consists of the group of diffeomorphisms of $S^3$ that leave the link invariant, module isotopy.  Even for knots, these symmetry groups include, for instance, all  dihedral groups.  

\smallskip
\noindent{\it Acknowledgments}   I have benefited from comments from Ryan Budney, whose work provides a backdrop for our approach. Nathan Dunfield provided an example of a four-component link $L$ with $\S(L) = \A_4$.  I was also helped by discussions with   Jim Davis, Allan Edmonds, Charlie Frohman, Michael Larsen,  Swatee Naik and Dylan Thurston.   


\section{The general setting of oriented links}\label{sec:generalities}
We now describe the general theory of {\it intrinsic symmetry groups} of links.  This theory was initially developed by Fox and was first presented by   Whitten in~\cite{MR242146}.   To be precise, we will momentarily consider links in  3--manifolds that are diffeomorphic to $S^3$, rather than work specificially with $S^3$.   In this setting we have the following definition: an {\it $n$--component link}\ is an  ordered  $(n+1)$--tuple of oriented manifolds,  $L = (S, L_1, L_2, \ldots , L_n)$, where $S$ is diffeomorphic to $S^3$  and the $L_i$ are disjoint submanifolds of $S$, each diffeomorphic to $S^1$.  The set of $n$--component links will be denoted $\call_n$.

Given a second link $L' = (S', L_1', L_2', \ldots , L_n')$, an {\it orientation preserving diffeomorphism} from $L$ to $L'$ is an orientation preserving diffeomorphism $F \co S \to S'$ such that $F(L_i) = L'_i$ as oriented manifolds  for all $i$.  

For any oriented manifold $M$, $-M$ denotes its orientation reverse. Let $\Z_2$ be the cyclic group of order two written multiplicatively:   $\Z_2 = \{1, -1\}$.  If $\epsilon = -1 \in \Z_2$, we will let $\epsilon M = -M$, and if $\epsilon = 1 \in \Z_2$, we will let $\epsilon M = M$.  
The group $\Z_2 \oplus (\Z_2)^n$   acts on $\call_n$ by  changing the orientations of the factors.  The symmetric group $\Sn$  acts on $\call_n$ by permuting the component knots.  These actions do not commute, but together define an action on the set of knots by the    {\it Whitten group}: 
\[
\Gamma_n = \Z_2 \oplus \big(     (\Z_2)^n \rtimes \Sn \big).
\]
In this semi-direct product, $\Sn$ acts on the $n$--fold product by permuting the coordinates. 
To be precise, given  an element $s = \big(\eta, (\epsilon_1, \ldots , \epsilon_n), \rho \big) \in \Gamma_n$ and an $n$--component link  $L$, we let 
\[ sL =( \eta S,  \epsilon_1 L_{\rho(1)}, \cdots , \epsilon_n L_{\rho(n)}) .\]  Notice that these  group actions are defined to be on the  left.  Thus, elements in $\Sn$ are multiplied right-to-left.  

\begin{definition} For a link $L \in \call_n$,   the {\it intrinsic symmetry group} of $L$ is the subgroup $\Sigma(L) = \{ s \in \Gamma_n\ \big| \ sL \cong L\} \subset \Gamma_n$.  Note that ``$\cong$'' indicates the existence of an orientation and order preserving diffeomorphism.  

\end{definition}

There are two fundamental questions regarding such link symmetries. \vskip.05in

\noindent{\bf Problem 1.}  Given    an $n$--component  link $L$, determine  $\Sigma(L)$.\vskip.05in

\noindent{\bf Problem 2.}  For each subgroup $H \subset \Gamma_n$, does there exist an $n$--component  link $L$ such that  $\Sigma(L) = H$?

\smallskip

The first   can be effectively answered for low-crossing number links with programs such as Snappy~\cite{SnapPy}.  The second is the focus of this paper; we  present the first examples of groups that cannot arise as the symmetry group of a link.

\subsection{Restricting to the oriented category and basic observations}  

There is a canonical index two subgroup $\overline{\Gamma}_n \subset  {\Gamma}_n$ consisting of elements of the form 
\[ (1, (\epsilon_1, \ldots, \epsilon_n), \rho).\]  This subgroup maps onto $\S_n$.  We leave it to the reader to verify the following, which   implies that any  constraint on what groups occur as $\S(L)$ places a constraint on what groups can arise as $\Sigma(L)$. 

\begin{theorem}  The image of $\Sigma(L) \cap \overline{\Gamma}_n$ in $\S_n$ is precisely $\S(L)$. \end{theorem}

After the  initial sections of this paper, we will be  restricting our  work to orientation preserving diffeomorphisms of $S^3$ and will work with unoriented links.      We will use the  following conventions which were summarized in the introduction.

\begin{enumerate} 
\item Links will all be of the form   $L = (S^3, L_1, L_2 , \ldots ,L_n)$ where $S^3$ has some fixed  orientation and the $L_i$ are disjoint unoriented submanifolds, each diffeomorphic to $S^1$. 
\item We will consider diffeomorphisms of the   link that are   orientation preserving on $S^3$ and that possibly permute the set of $L_i$. 
\item The set of such diffeomorphisms  will be   denoted $\diff(L)$.

\item Given  $F \in \diff(L)$ we have $(S^3, F(L_1), F(L_2), \ldots , F(L_n)) = (S^3,  L_{\rho(1)} , L_{\rho(2)} \ldots ,  L_{\rho(n)})$ for some $\rho \in \S_n$.  This defines a homomorphism $\Phi \co \diff(L) \to \Sn$.

\item The image $\Phi$  in $\S_n$ is denoted $\S(L)$.

\end{enumerate}



\section{Examples: knots}\label{sec:knots} 

Before restricting to the orientation preserving diffeomorphism  group, in this section and the next we will summarize what is known in general for links of one and of two components. Then, in Section~\ref{sec:fully}, we show that for all $n$ there is a prime, non-splittable link with 
$\Sigma(L) = \Gamma_n$.   

Let  $n=1$.    The symmetric group $\S_1$ is trivial and thus the first Whitten group is   $\Gamma_1 \cong \Z_2 \oplus \Z_2$.  The knots $(1,-1)K, (-1,1)K,$ and $(-1,-1)K$ have been called the {\it reverse}, $K^r$, the {\it mirror image}, $m(K)$, and the {\it reversed mirror image}, $m(K)^r$, respectively. (Older references have called the reverse of $K$ the {\it inverse}.  The name ``reverse'' is used to distinguish it from the concordance inverse, which is represented by the reversed mirror image.)  Figure~\ref{fig:trefoils} illustrates the possibilities.  A detailed account of the key results in the study of knot symmetries is contained in~\cite{MR867798}.  Here is a brief summary.

The group $\Gamma_1$ has five subgroups:  the entire group, the trivial subgroup, and the three subgroups containing exactly one of the nontrivial elements of $\Gamma_1$.  Each is realized as   $\Sigma(K)$ for some knot $K$.  

\begin{itemize}
\item  The unknot or the figure eight knot, $4_1$, have full symmetry group.  They are called {\it fully amphicheiral}.

\item  The trefoil knot is   reversible.   Dehn showed that it does not equal its mirror image, a fact that can now be proved using such invariants as the signature or the Jones polynomial.  Thus, $3_1$ is {\it reversible}. 

\item  Trotter~\cite{MR0158395} proved the existence of non-reversible knots.  His examples in~\cite{MR0158395}  have nonzero signature  and thus have trivial symmetry group.  We say that such knots are  {\it chiral}.  Hartley~\cite{MR683753}  proved that $9_{32}$ is non-reversible and since it has non-zero signature, it too is chiral.

\item Kawauchi~\cite{MR559040} proved that $K = 8_{17}$ is non-reversible.  It is easily seen that $K = m(K)^r$, and thus $8_{17}$ is {\it negative amphicheiral}.

\item The simplest example of a low-crossing number knot that is   non-reversible and for which $K = m(K)$ is $12a_{147}$, which was detected by the program Snappy.  (Presumably the general techniques developed by Hartley in~\cite{MR683753} would also show that this knot is not reversible.)   More complicated examples of such {\it positive amphicheiral} knots were   first discovered by Trotter.

\end{itemize}

\begin{figure}[h]
\labellist
\pinlabel {\text{{$K$}}} at 100 -35
\pinlabel {\text{{$(1,-1)K = K^r$}}} at 500 -35
\pinlabel {\text{{$(1,-1)K = m(K)$}}} at 800 -35
\pinlabel {\text{{$(-1,1)K = m(K)^r$}}} at 1150 -35
\endlabellist
\includegraphics[scale=.3]{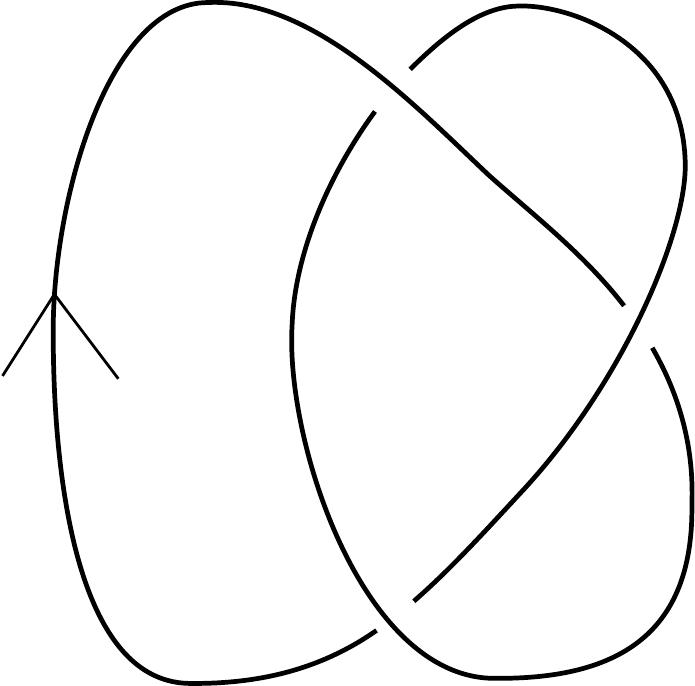} \hskip.6in \includegraphics[scale=.3]{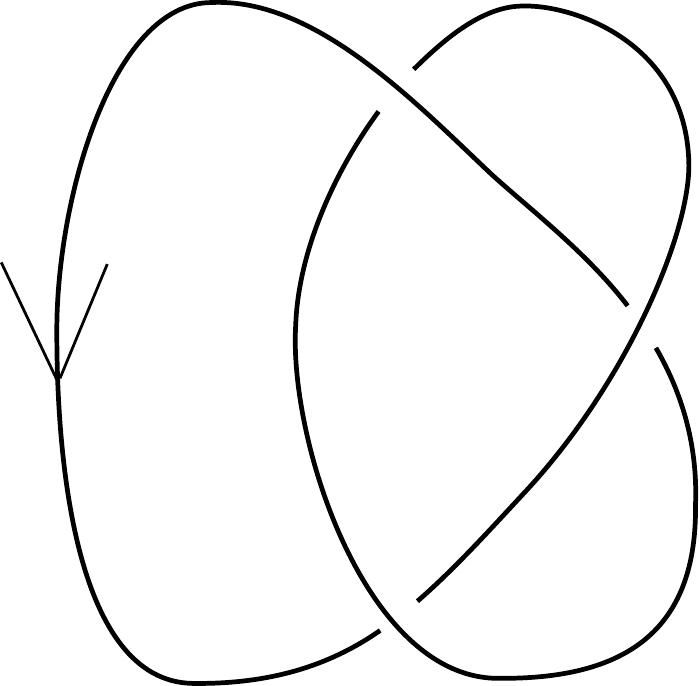} \hskip.6in 
\includegraphics[scale=.3]{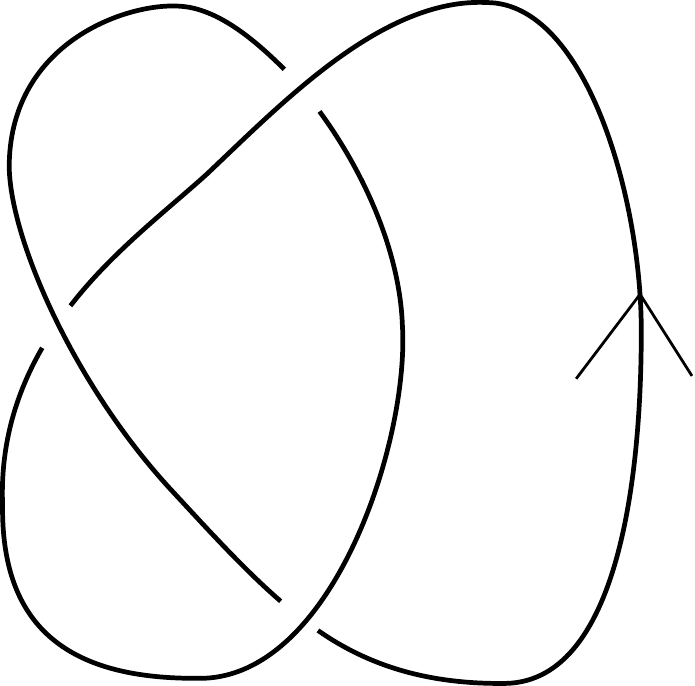} \hskip.6in   \includegraphics[scale=.3]{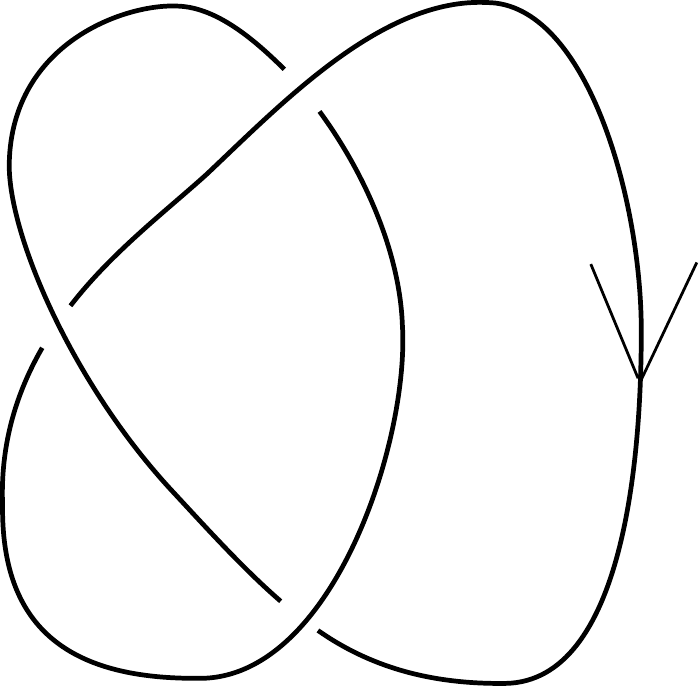} 
\vskip.15in
\caption{Symmetries of knots.}
\label{fig:trefoils}
\end{figure}


\section{Two-component links}\label{sec:links}

Here we summarize the results of~\cite{MR2909632, MR2909631} concerning two-component links. We have that $\Gamma_2  = \Z_2 \oplus \big(     (\Z_2)^2 \rtimes \S_2 \big)$ is of order 16.  In~\cite{MR2909632, MR2909631}  the authors describe the 27 conjugacy classes of subgroups of $\Gamma_2$.  They then show that tables of prime, non-splittable links provide examples of links realizing 21 of these subgroups.   One of the missing subgroups is $\Gamma_2$ itself.  This is clearly the symmetry group of the unlink; in a note on MathOverflow~\cite{rbudney}, Budney showed that $\Gamma_2$ is the symmetry group of a non-splittable  Brunnian link.  We will expand on that example in the next section.

To conclude this section, we list the subgroups that are currently not known to be the symmetry groups of two-component links.  Let $\tau$ denote the transposition in $\S_2$.

\begin{itemize}
\item $\left<   \      (1, (-1,1)\  \tau   )                 \right> \cong \Z_4$. \vskip.05in

\item $\left<     \    (1, (-1,1))\    ,\ (-1, (1,1)) \                  \right> \cong \Z_2 \oplus \Z_2$. \vskip.05in

\item $\left<      \   (1, (1,-1))\    ,\ (-1, (-1,1))   \                \right> \cong \Z_2 \oplus \Z_2$. \vskip.05in

\item $\left<       \  (-1, (-1,1)) \   ,\ (1, (-1,1)\tau)  \                 \right> \cong D_4$, the dihedral group with four elements. \vskip.05in

\item $\left<       \  (1, (1,-1)) \   ,\ (1, (-1,1))  \   ,\  (-1, (1,1)) \               \right> \cong \Z_2 \oplus \Z_2 \oplus \Z_2$.\vskip.05in

\end{itemize}

\section{Fully amphicheiral links for all $n$} \label{sec:fully} In Figure~\ref{fig:companion} we illustrate a knot $K$ in a solid torus $D$.  Two parallel strands of $K$ are tied in a  knot $J$, where $J$ is chosen to be fully amphicheiral; the figure eight knot would be sufficient.  As oriented pairs, we have $(D,K) \cong (-D,K) \cong (D, -K) \cong (-D, -K)$.

\begin{figure}[h]
\labellist
\pinlabel {\text{\Large{$ J$}}} at 180 58
\endlabellist
\includegraphics[scale=.32]{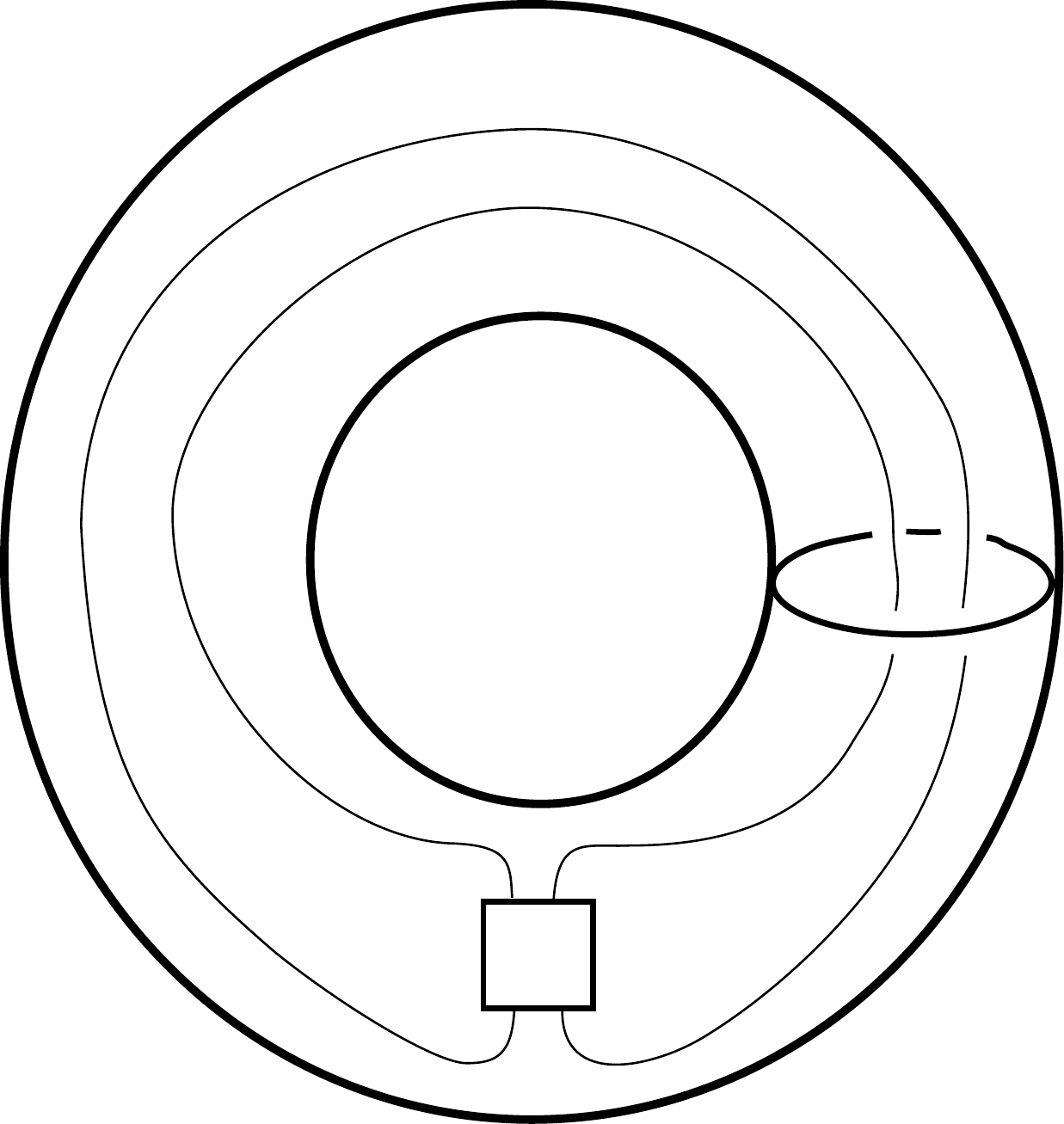}  
\vskip.15in
\caption{Companion.}
\label{fig:companion}
\end{figure}

Budney's example~\cite{rbudney} of a two-component link $L$ with full symmetry group $\Sigma(L) = \Gamma_2$ is formed from the Hopf link by replacing neighborhoods of each component with copies of $(D,K)$.  An example of a three-component link with full symmetry group is built in the same way, starting with the Borromean link.  Notice that in both these examples, the links are Brunnian.   
Problem~\ref{prob:brunnian} in Section~\ref{sec:questions} asks:   Does there exist a  Brunnian link with  four or more components with full symmetry group? 

We conclude this section with an elementary observation.

\begin{theorem}\label{thm:fullsym}  For every $n$ there exists a   prime, non-splittable  link $L$ for which $\Sigma(L) \cong \Gamma_n$.
\end{theorem}
\begin{proof}
To form an $n$--component link with full symmetry group, proceed as follows.  Starting with the knot $J$, form a link by replacing $J$ with $n$--parallel copies of $J$; formally, form the $(n,0)$--companion of $J$.  Next, replace a neighborhood of each component of that link with a copy of $(D,K)$.  Innermost circle arguments, dating to the work of Schubert~\cite{MR72482}, can be used to show that this link $L$ is prime and   non-splittable. 
\end{proof} 


\section{Torus decompositions and tree diagrams} \label{sec:jsw-trees} A principal tool in understanding knot and link complements is the Jaco-Shalen-Johannson torus decomposition, which we refer to as the JSJ--decomposition.  An excellent resource is ~\cite{MR2300613} which contains details for the results we summarize here.

Let $X$ be the complement of a  non-splittlable link $L$ in $S^3$.  The JSJ-composition of $X$ is given by a finite family of disjoint incompressible embedded tori, $\{T_i\}$, with the property that each component of the complement of $\cup T_i$    has either a complete hyperbolic structure or is Seifert fibered.  There is the additional condition that no $T_i$ is boundary parallel and that no two of the $T_i$ are parallel.  Up to isotopy, there  is a unique minimal set  $\{T_i\}$ with these properties; this set provides the JSJ-decomposition.  No two $T_i$ in the decomposition are isotopic.

We can associate a finite tree  Tr($L$)  to this decomposition, as follows.  Let the components of $X \setminus \cup T_i$ be denoted $\{C_i\}$.  The vertices of the Tr($L$) correspond to the $C_j$.  Two vertices are joined by an edge if the  closures of the corresponding $C_i$ intersect;  there is one edge for each $T_i$.   When possible, we will use the names $C_i$ and $T_i$ to denote the vertices and edges.   We will say that a component $C_i$ {\it contains a component} $L_j \in L$ if $L_j$ is in the closure of $C_i$. 

\subsection{The subtrees   $ {\text{Tr}}_L(K) $  and $\widehat{\text{Tr}}(L)$}

Let $K$ be a component of $L$.  Its orbit under the action of $\diff(L)$ is a sublink  of $L$,  $ \{K_1, \ldots, K_m\}$, where $K_1 = K$. Each $K_i$ is contained in a vertex of $\text{Tr}(L)$.  The set of such vertices is denoted $\{D_1, \ldots , D_k\}$.  (Since the action of $\diff(L)$ on the set of $K_i$ is transitive, each $D_j$ contains the same number of components of $L$.  In particular, $k$ divides $m$.  Later we will expand on this observation.

The vertices $ \{D_1, \ldots , D_k\}$ in  $\text{Tr}(L)$ span a unique minimal subtree, which we denote $ {\text{Tr}}_L(K)$.  In the case that the action of $\diff(L)$ is transitive on $L$, the orbit of $K$ is all of $L$, and we write  $\widehat{\text{Tr}}(L) = {\text{Tr}}_L(K)$.   (Notice that $\widehat{\text{Tr}}(L)$ need not equal $T(L)$;  for instance, vertices of $T(L)$ of valence one that do not contain components of $L$ are not included in $\widehat{\text{Tr}}(L)$.)  

\begin{theorem} If   $\diff(L)$ acts transitively on $L$, then the tree $\widehat{\text{Tr}}(L)$ either contains exactly one vertex, or its valence one vertices are precisely the set  $ \{D_1, \ldots , D_k\}$.
\end{theorem}
\begin{proof}
It is an elementary observation that in the subtree of a tree spanned by the set of vertices $\{D_j\}$, the only vertices of valence one correspond to elements in the set $\{D_j\}$ and that if  there is more than one $D_j$, then at least one of them is a vertex of valence one.  We need to see that each $D_j$ has valence one.

Suppose that the vertex $D_1$ is of valence one in  $\widehat{\text{Tr}}(L)$ and that it contains $L_1$.  Let $D_2$ be another vertex, and suppose it contains $L_2$.  There is an element $F \in \diff(L)$ such that $F(L_1) = L_2$.  The map $F$ is isotopic relative to $L$ to a diffeomorphism $F'$ that preserves the JSJ-decomposition.  This $F'$ induces an automorphism of  $\text{Tr}(L)$ that leaves   $\widehat{\text{Tr}}(L)$ invariant.  Thus, there is an automorphism of  $\widehat{\text{Tr}}(L)$ that carries $D_1$ to $D_2$.  It follows that $D_2$ is of valence one in  $\widehat{\text{Tr}}(L)$.
  
\end{proof}

\subsection{The group $\diff^*(L)$.}  Fix a JSJ-decomposition of $S^3 \setminus L$.   

\begin{definition}  We let $\diff^*(L) \subset  \diff(L)$ to be the subgroup consisting of elements that leave the JSJ-decomposition invariant.

\end{definition}

\begin{theorem}The image of $ \diff^*(L) $ in $\Sn$ equals $\S(L)$.

\end{theorem}

\begin{proof} Given an element in $\S(L)$, there is a diffeomorphism $F\in \diff(L)$ that maps to it.  We have that $F$ is isotopic relative to $L$ to an element $F' \in \diff^*(L)$.  The map $F'$ induces the same permutation of the components of $L$ as does $F$.

\end{proof}

\begin{theorem} In the case that $\diff^*(L)$ acts transitively on the components of $L$,  the action of $\diff^*(L)$ on  $\widehat{\text{Tr}}(L)$ factors through an action of $\S(L)$ on   $\widehat{\text{Tr}}(L)$.
\end{theorem}

\begin{proof} An automorphism of a tree is completely determined by its action on the valence one vertices of the tree.  We leave this elementary observation to the reader. 
\end{proof}

\subsection{The structure of  $\widehat{\text{Tr}}(L)$ when $\S(L) = \A_n$  }
In Figure~\ref{fig:tree1} we provide  an example of a labeled tree to serve as a model for the discussion that follows.

\begin{figure}[h]
\labellist
\pinlabel {\text{\Large{$D_1$}}} at  20 20
\pinlabel {\text{\Large{$D_2$}}} at 86 20
\pinlabel {\text{\Large{$D_3$}}} at 152 20
\pinlabel {\text{\Large{$D_4$}}} at 218 20
\pinlabel {\text{\Large{$D_5$}}} at 285 20
\pinlabel {\text{\Large{$D_6$}}} at 349 20
\pinlabel {\text{\Large{$C_1$}}} at 93 148
\pinlabel {\text{\Large{$C_2$}}} at 290 148
\pinlabel {\text{\Large{$C_0$}}} at 188 244
\endlabellist
\includegraphics[scale=.48]{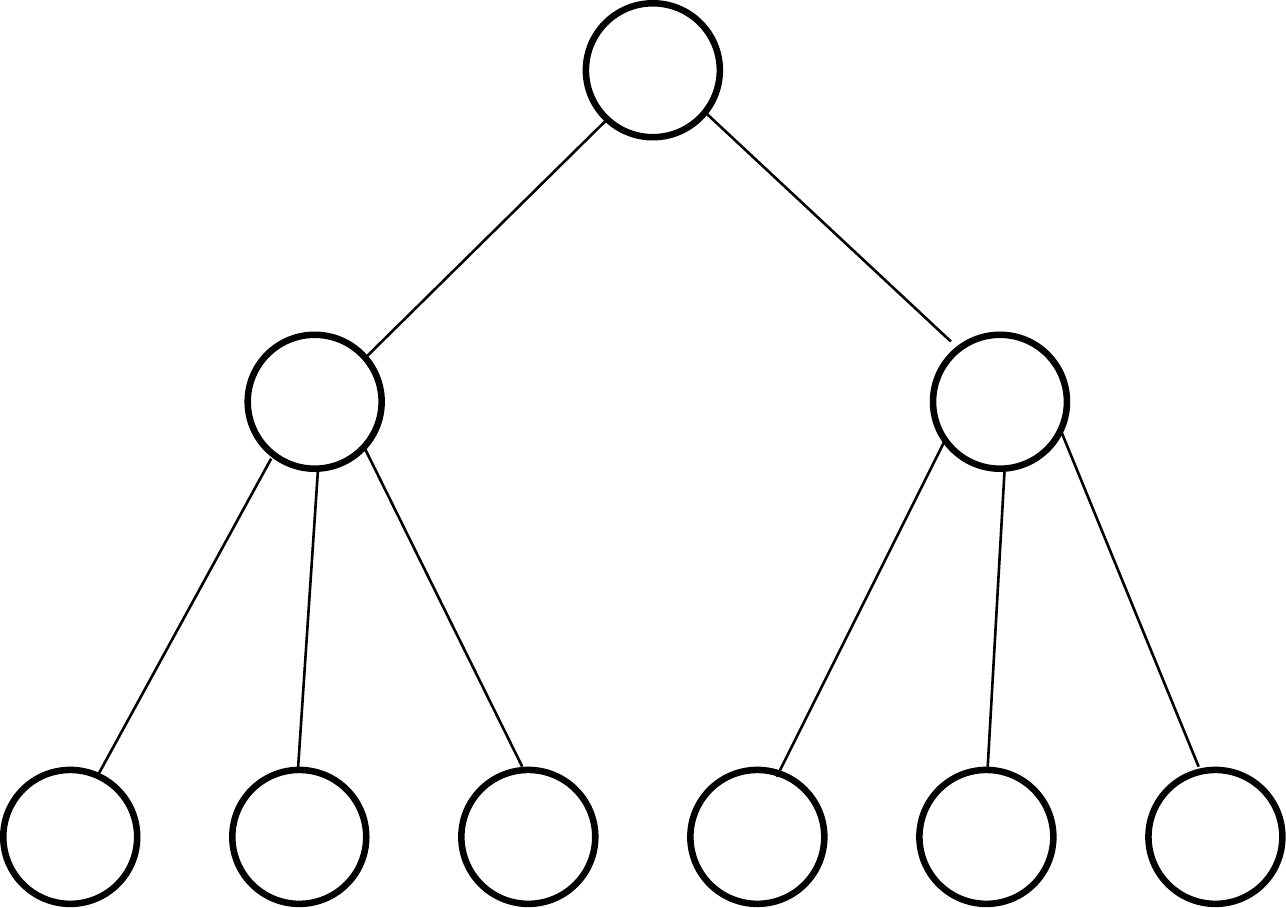}  
\vskip.15in
\caption{Tree diagram for sublink $K$ of $L$ on which $\diff(L)$ acts transitively.}
\label{fig:tree1}
\end{figure}
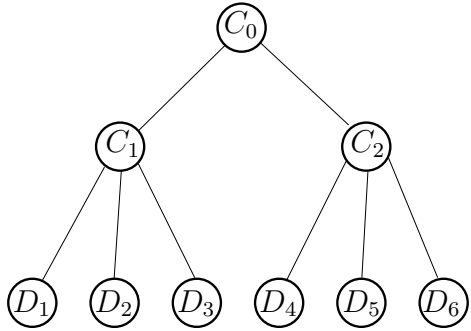

\begin{lemma}  If   $\S(L) = \A_n$   and   $\widehat{\text{Tr}}(L)$ contains more than one vertex, then each $D_i$ contains exactly one $L_1$ and the number  of vertices in the set   $\{D_i\}$ is $n$.

\end{lemma} 

\begin{proof} Suppose that $  D_1$  contains $L_1$ and $L_2$ and   that $D_2$ contains   $L_3$ and $ L_4$.  Then the permutation $(1 2 3) \in \A_n$ does not induce an action on  $\widehat{\text{Tr}}(L)$.

\end{proof}

\begin{theorem}\label{thm:alt2} If $\S(L) = \A_n$, $n\ge 3$, then  $\widehat{\text{Tr}}(L)$  is a rooted tree with either exactly one vertex, $C$,  or with   $n$    vertices of valence one. In the second case,   there is a  unique vertex with valence greater than 2; the tree $\widehat{\text{Tr}}(L)$ is built from that high valence   vertex $C$ by attaching $n$ linear branches, all of the same length.  The vertex $C$ is invariant under the action of $\A_n$ on   $\widehat{\text{Tr}}(L)$.
\end{theorem}

\begin{proof}

Figure~\ref{fig:tree3} is a schematic of a tree.  We are asserting that  $\widehat{\text{Tr}}(L)$ is of this form.

\begin{figure}[h]
\labellist
\pinlabel {\text{\Large{$D_1$}}} at  20 20
\pinlabel {\text{\Large{$D_2$}}} at 86 20
\pinlabel {\text{\Large{$D_3$}}} at 152 20
\pinlabel {\text{\Large{$D_4$}}} at 218 20
\pinlabel {\text{\Large{$D_5$}}} at 285 20
\pinlabel {\text{\Large{$C$}}} at 153 242
\endlabellist
\includegraphics[scale=.48]{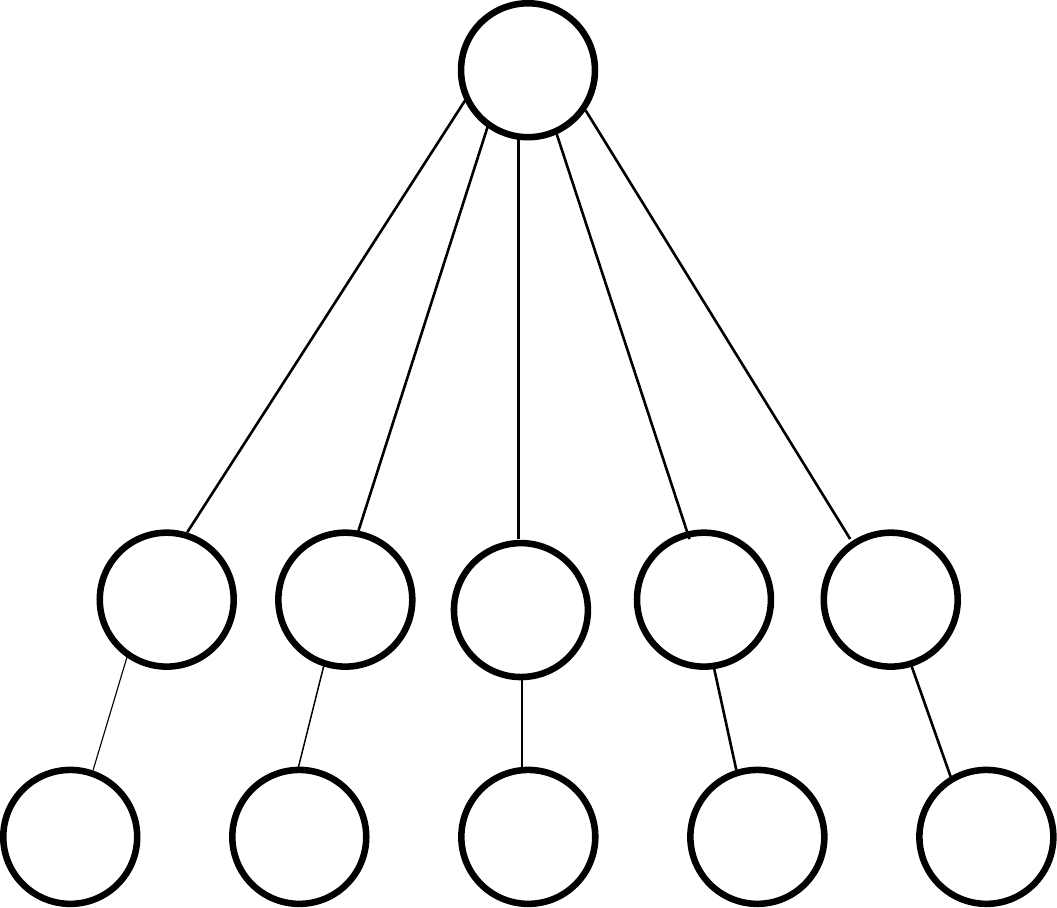}  
\vskip.15in
\caption{Possible tree diagram $\widehat{\text{Tr} } (L)$ for a five component link $L$ on which $\S(L) = \A_5$.}
\label{fig:tree3}
\end{figure}
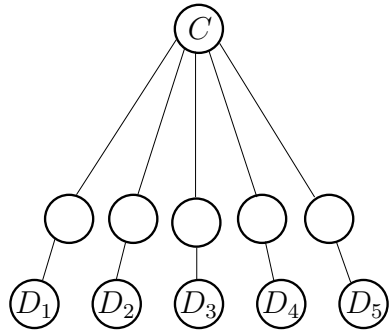

We have seen that each $D_i$ contains precisely one $L_i$ and these are the valence one vertices of $\widehat{\text{Tr} } (L)$.  
A tree with more than two valence one vertices always contains some vertex with valence greater than 2.  It remains to show that there is a unique such vertex of valence greater than two.  (For an example of the sort of tree we need to rule out, build a tree from two copies of the graph illustrated in Figure~\ref{fig:tree3}  by joining the roots with  single edge.)

An elementary exercise shows that for any tree on which $\A_n$ acts and  for which the action of $\A_n$  on  vertices of valence one is transitive,  there is an invariant vertex or edge:  proceed by induction,  removing all valence one vertices and their adjacent edges from the tree.  

We next observe that in the case that the symmetry group is $\A_n$, there must be an invariant vertex.  The action of the symmetry group of the tree is transitive on its valence one vertices, so if there is an invariant edge, some elements must reverse that edge.  It follows that the subgroup of the symmetry group that does not reverse the edge is index two.  But $\A_n$ does not contain an index 2 subgroup for $n\ge 3$. 

\end{proof}

\subsection{The structure of the core $C$ in the case that $\S(L) = \A_n$.}

Suppose that $\S(L) = \A_n$.  Then by Theorem~\ref{thm:alt2} there is a core $C$ in the JSJ--composition of $L$.  This core is acted on by $\diff^*(L)$.  The boundary of $C$ is the union of two sets of tori, $\{T_1, \ldots , T_n\} \cup  \{S_1, \ldots , S_m\}$.  Each $T_i$ bounds a submanifold $W_i \subset S^3$ that contains the link component $L_i$  and does not contain $C$.   A schematic appears in Figure~\ref{fig:tree2}.  In this diagram we have included extra edges showing   $ \widehat{\text{Tr}}(L)$  might be a proper subtree of   $ {\text{Tr}}(L)$  and that $C$ might have more than $n$ boundary components.

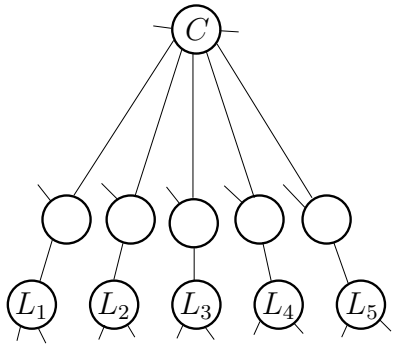
\begin{figure}[h]
\labellist
\pinlabel {\text{\Large{$L_1$}}} at  20 30
\pinlabel {\text{\Large{$L_2$}}} at 86 30
\pinlabel {\text{\Large{$L_3$}}} at 152 30
\pinlabel {\text{\Large{$L_4$}}} at 218 30
\pinlabel {\text{\Large{$L_5$}}} at 285 30
\pinlabel {\text{\Large{$C$}}} at 153 252
\endlabellist
\includegraphics[scale=.48]{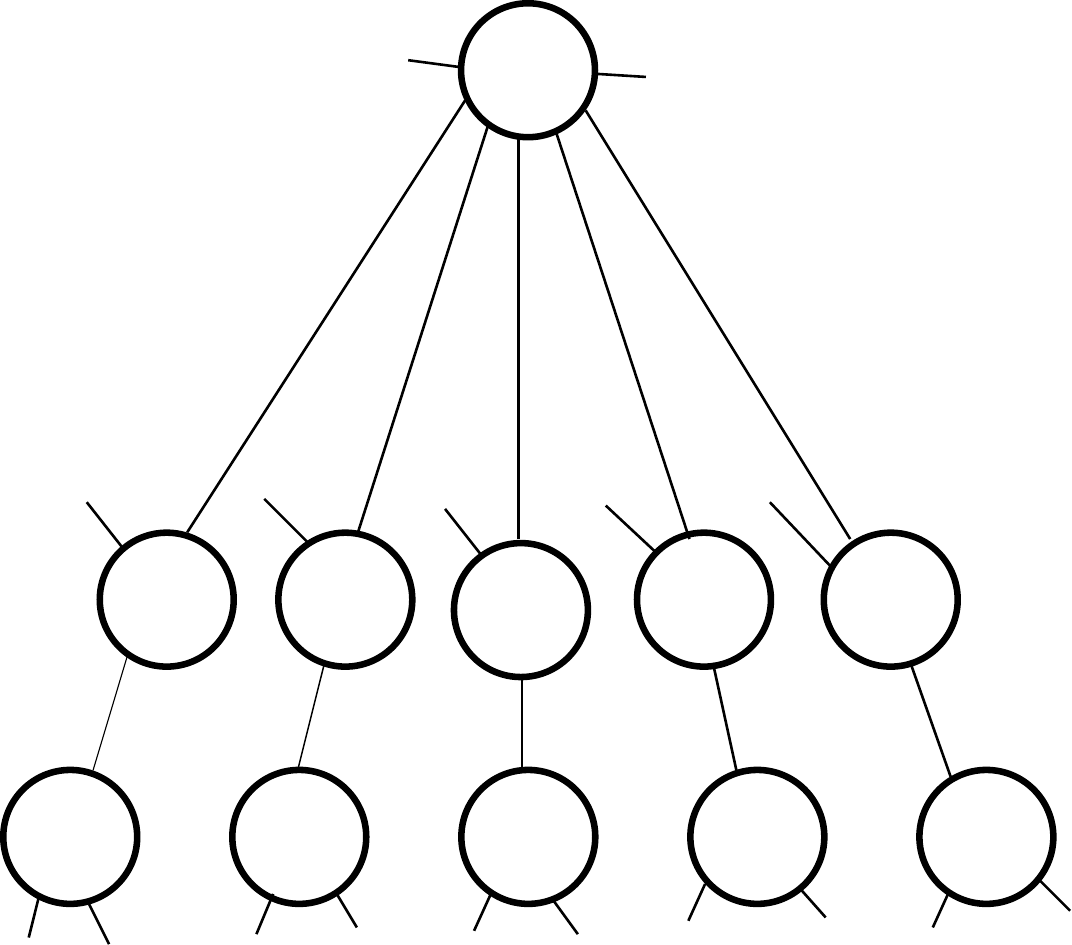}  
\vskip.15in
\caption{Possible tree diagram $\widehat{\text{Tr} } (L)$ for a five component link $L$ on which $\diff^*(L)$ acts transitively.}
\label{fig:tree2}
\end{figure}

Let $\diff(C)$ be the diffeomorphism group of the core $C$.  It contains a subgroup $\diff(C,T)$ that leaves invariant the set of  $T_i$.  This group maps to $\S_n$ via its action on $\{T_i\}$. 

\begin{theorem}  In the case that $\S(L) = \A_n$ with $n\ge 5$, with core $C$, the group $\diff(C,T)$ acts on $\{T_i\}$ as either the  $\A_n$ or $\Sn$.  
\end{theorem}

\begin{proof}  It is clear that the action contains $\A_n$.  The only subgroups of $\Sn$ that contain $\A_n$ are $\A_n$ and $\Sn$.
\end{proof} 

Notice that it might happen that there are elements of $\diff(C,T)$ that do not map to elements of $\A_n$; it is possible that not every action on $C$ extends to $S^3$.


\section{Reembeddings}\label{sec:reembed}

Reembeddings appear in two different ways in our proof.  In the case of $C$ hyperbolic, we embed $C$ in $S^3$ as a link complement.  In the Seifert fibered case, we embed $C$ into a closed Seifert fibered space as the complement of a set of regular fibers.  In this section, we describe the embedding into $S^3$.

In the previous section, some of the (torus) boundary components of the core  $C$ were denoted $T_i$.   We will now see that by using reembeddings we can view these $T_i$, along with the other boundary components $S_i$ of $C$, as peripheral tori  for a link in $S^3$.  This is presented in~\cite{MR2300613}, where Budney gave a reembedding theorem for submanifolds of $S^3$.  Here we present a slightly enhanced version of that result, keeping track of boundary curves.  First we set up some notation.

Let $X\subset S^3$ be a compact connected submanifold with one of its boundary components a torus $T$.   The complement of $T$ consists of two spaces, $Y_1$ and $Y_2$.  We have $H_1(Y_1) \cong \Z \cong H_1(Y_2)$.  We assume $X \subset Y_1$.   When needed, we will write these as $Y_1(X,T)$ and $Y_2(X,T)$.

We have that $\ker( H_1(T) \to H_1(Y_1)) \cong \Z$.  The generator can be represented  by a simple closed curve we denote $l$.   Similarly, a representative of  $\ker( H_1(T) \to H_1(Y_2)) \cong \Z$ is denoted $m$.  There is no natural orientation for these choices.  However, we can assume that they are oriented so that  the intersection number of $m$ and $l$ is 1 with respect to the orientation of $T$ viewed as the boundary of $Y_1$.  We can also assume that $m$ and $l$ intersect transversely in exactly one point.  With this setup, we have the following.

\begin{theorem} There exists an orientation preserving embedding $F\co X \to S^3$ such $ F(T)$ is the boundary of a   tubular neighborhood of a knot in $S^3$ having meridian $F(m)$ and longitude $F(l)$.  
\end{theorem}

\begin{proof}
An embedded torus in $S^3$ bounds (on one side or the other) a solid torus which we denote $W$.   If $Y_2 = W$, then  $m$ is the meridian of $W$ and $F$ can be taken to be the identity.  

If $Y_1 = W$, then form the   boundary union $Z = Y_1 \cup W'$, where $W'$ is a solid torus, attached so that its meridian is identified with $m$ and its longitude is identified with $l$.  Then $Z$ is the union of two solid tori and the choice of identification ensures that $H_1(Z) = 0$.  Thus, $Z \cong S^3$.   
\end{proof}

\begin{corollary} \label{cor:embed2}Suppose that $X \subset S^3$ is a compact manifold with boundary a union of tori $\{T_1, \ldots , T_k\}$.  There exists a link $L = \{L_1, \ldots , L_k\}$ and an orientation preserving homeomorphism  $F\co X\to  S^3 \setminus \nu(L)$, where $\nu(L)$ is an open tubular neighborhood.  Furthermore, it can be assumed that $F$ preserves meridians and longitudes.\end{corollary}

\begin{corollary} \label{cor:embed3} With $X \subset S^3$  and $L$  as in Corollary~\ref{cor:embed2},  suppose that a diffeomorphism $g\co S^3 \to S^3$ satisfies $g(X) = X$.   Then the diffeomorphism of $F(X)$ given as the composition $F \circ g \circ F^{-1}$ extends to a diffeomorphism of $(S^3, L)$.

\end{corollary}

\noindent{\bf Note.} Not every diffeomorphism of $X$ determines a diffeomorphism of $L$.  It is essential here that the diffeomorphism of $X$ extends to $S^3$.

\subsection{Summary theorem}

\begin{theorem}  Suppose that $\S(L) = \A_n$.  Then there is a link $( L'_1, \ldots, L'_n, J_1, \ldots J_m)$ with complement diffeomorphic to  $C$  and that is either hyperbolic or Seifert fibered.  The mapping class group of this link has a subgroup that preserves $(L_1',  \ldots L_n')$.  The image of this subgroup in $\Sn$ is either $\A_n$ or $\Sn$.
\end{theorem}

\begin{proof}  To prove this using the previous results, we need to show that a JSJ-decomposition exists.  That is, that $L$ is non-splittable.   If $L$ does split, it splits as  the union non-split sublinks, say $D_1, \ldots , D_k$, where each $D_i$ is contained in a ball that does not intersect the other $D_j$.  The transitivity of the $\A_n$--actions implies that the $D_i$ are identical links.  Thus, we can write $D_i = \{ D_i^1, \ldots , D_i^m\}$ for some $m$ that is independent of $i$.  

If $m =1$, then $L$ is consists of $n$ copies of a knot $J$, each copy in a separate ball.   In this case, the symmetry group would be $\Sn$.   If $m =n$, then we are in the nonsplit case, as desired.

Finally, if $1 < m < n$, then any element of $\calm$ that carries $D_1^1$ to $D_2^1$, must carry $D_1^2$ to some $D_2^i$.  But not every element of $\A_n$ behaves in this way.

\end{proof}

To complete the proof that $\A_n$, $n \ge 6$, is not the intrinsic symmetry group of any link, we consider the hyperbolic and Seifert fibered cases separately.  


\section{The case of $C$ hyperbolic.}\label{sec:hyperbolic}

We use the notion of {\it core} as in the previous section.  

\begin{theorem}If $\A_n \subset \S(L)$  and the core $C$ is hyperbolic, then  some finite subgroup of $SO(4)$ contains a finite subgroup having $\A_n$ as a quotient.
\end{theorem}

\begin{proof}

For each element $\phi \in \diff(C, \partial C)$ that extends to $S^3$, let $\phi'$ denote an isometry that is isotopic to $\phi$ relative the boundary.  Note that the actions of $\phi$ and $\phi'$ on the finite set of components   $\{\partial C\}$ are the same.  The set of $\phi'$ generates a subgroup of Isom($C$).  This is necessarily a finite group, $H$.  The group $H$ contains the subgroup $H' \subset H$ that leaves invariant the set $\{L_1', \ldots , L_n'\}$.   The image of $H'$ in $\Sn$ contains $\A_n$.  By restricting to a further subgroup $H''$, we can assume the image is precisely $\A_n$.

By results such as~ \cite{MR2416241,  MR2178962}, any finite subgroup of Diff($S^3$), such as $H''$,  is    isomorphic to a subgroup of $SO(4)$.                                                                                                                                                                                                                                                                                                                                                                                                                                                                                                                                                       

\end{proof}

\begin{corollary} If $\A_n \subset \S(L)$ then $n \le 5$.
\end{corollary}

\begin{proof}   This follows from the results of the next subsection. \end{proof}

\subsection{The only subgroup of $SO(4)$ that maps onto a  noncyclic simple group is isomorphic to $\A_5$.}

We prove somewhat more than this.

\begin{theorem}  If $A$ is a nonabelian simple group and a subgroup $H \subset SO(4)$ surjects onto $A$, then $A\cong \A_5$.
\end{theorem} 

Denote the surjection from $H$ to $A$ by $\phi\co H \to A$.  We begin by recalling the structure of $SO(4)$.

The set of unit quaternions is  homeomorphic to $S^3$ and  as a Lie group  is isomorphic   to $SU(2)$.  Quotienting by $\pm 1$ yields a 2--fold cover $SU(2) \to SO(3)$. 

Let $x$ and $y$ be unit quaternions and view elements $v \in \R^4$ as quaternions.  Then $x$ and $y$ define a homomorphism $\psi_{x,y} \co SU(2) \times SU(2) \to SO(4)$ by  $\psi_{x,y} (v) = xvy^{-1}$.  This yields a 2--fold covering of   $SU(2) \times SU(2) \to SO(4) $.   Hence,    $SO(4) \cong (SU(2)) \times SU(2))  / < (-1,-1)> $.

There is a 2--fold covering space  $q\co (SU(2)) \times SU(2)/ \left< (-1,-1)\right>  \to SO(3) \times SO(3)$.   We thus have the following diagram.  

\[
\begin{diagram}
\node{(SU(2) \times SU(2))/ \left<(-1,-1)\right>} \arrow[2]{e,t}{\cong} \arrow{s,l}{2-\text{fold cover}, \text{$q$}}  
\node[2]{SO(4)} \\
\node{SO(3) \times SO(3)} \end{diagram}
\]

We will write element of $SO(4)$ and of $SO(3)\times SO(3)$ as equivalence classes of pairs of unit quaternions. 

\begin{lemma}  The map $\phi$ induces a surjection  $\phi' \co q(H) \subset SO(3) \times SO(3) \to A$.\end{lemma}
\begin{proof}  If the map $q\co H \to q(H)$ is an isomorphism, then this is trivially true.  It is possible that  $q\co H \to q(H)$ is two-to-one, which can occur if and only if the central element $(1,-1) \in H$. In this case $q(H) \cong H/\left< (1,-1)\right>$.  Since $A$ is nonabelian and simple, the image of $(1,-1)$ in $A$ is trivial.   \end{proof}

\begin{lemma}  Let $G \subset SO(3) \times SO(3)$.  Let $G_1$ and $G_2$ be the images of the projections of $G$ onto the first and second factors of the product.   If $\phi' \co G \to A$ where $A$ is nonabelian and simple, then a subgroup of  $G_1$ or $G_2$ maps onto $A$.  In particular, $A$ is a quotient of a finite subgroup of SO$(3)$.  \end{lemma}

\begin{proof} Let $F = G \cap ( \text{SO}(3) \times \{1\})$.  We have that  $F$ is a normal subgroup of $G$ and thus, $\phi'(F) = A$ or $\phi'(F) = \{1\}$.   In the first case, we are done, so assume that $\phi'(F) = \{1\}$.

We now define a surjective  homomorphism $\phi''\co G_2 \to A$.  Given $y \in G_2$, there exists an  element $x \in G_1$ such that $(x,y) \in G$.  Set $\phi''( y) = \phi'( (x,y))$.   To see that this is well-defined, notice that $(x_1, y) \in G$ and $(x_2, y) \in G$ then $x_1 x_2^{-1} \in F$.  Thus  $  \phi'( (x_1 ,y)) =   \phi'( (x_2 ,y)) $.  It is easily checked that $\phi'$ is surjective and is a homomorphism. 

\end{proof}
\begin{lemma}  The group $\A_5$ is the only finite  noncyclic simple  group contained in SO($3$).
\end{lemma}

\begin{proof}
The finite subgroups of SO$(3)$ are classified.  Here is the list of possibilities. 
\begin{itemize}
\item  $A_n \cong \Z_n$:  cyclic groups. 
\item $D_n$: dihedral groups. 
\item $E_6 \cong \A_4$: tetrahedral group.  
\item $E_7 \cong \S_4$: octahedral group.  
\item $E_8 \cong \A_5$: icosahedral group.  

\end{itemize}
The subgroups of the dihedral group are either dihedral, and thus not simple, or cyclic.  The smallest nonabelian simple group is $\A_5$.
\end{proof}


\section{The case of $C$ Seifert fibered.}\label{sec:seifertfibered}

We begin with a basic example.

\begin{example} \label{ex:sfexamp}Consider the  $(n+2)$-component link  $L$ formed as follows.  Let $T$ be a standardly embedded torus in $S^3$ and form the     $(np, nq)$--torus link on $T$ with $q > p >1$ relatively prime.   Add to this the cores of two solid tori bounded by $T$.  There is a Seifert fibration of $S^3$ with  the torus link represented by regular fibers and the two cores being neighborhood  of singular fibers of type $p/q$ and $q/p$.\

We leave it to the reader to confirm that for this link, $\Sigma(L) \cong \Z_2 \oplus \S_{n}$.  It should be clear how the components of the  $(np, nq)$--torus link  can be freely permuted.  The $\Z_2$ arises from a diffeomorphism that reverses all the components. 

Two exercises arise here.  The first is to show that   that every symmetry   fixes   the two core circles.  The second is to show that the complement of this link is homeomorphic to the complement of $n+2$ fibers of the Hopf fibration of $S^3$.

More examples can be built from this one.  Let $J\subset S^1 \times B^2$ be a knot for which $\partial (S^1 \times B^2)$ is incompressible in the complement of $J$.  A new link can be formed by replacing neighborhoods of the components of $L$  with copies of $S^1 \times B^2$.  Then the symmetry group of this new link will be isomorphic to either   $  \Z_2 \oplus \S_{n-2}$  or  $  \S_{n-2}$, depending on the symmetry type of $J$.

\end{example}  

\subsection{$C$ is the complement of regular fibers in a closed Seifert manifold}

\begin{example}  Figure~\ref{fig:seifertfibered} provides a schematic of one possible case in which the core $C$ is Seifert fibered.  Some of the  labels in the diagram will be explained later.  A link $L$ can be  formed by filling each $T_i$ with pairs $(S^1 \times B^2, J_i)$ and the $S_i$ are filled with either solid tori  or nontrivial knot complements.   There are constraints required for this to produce a link in $S^3$ and we  do not assert that  in all cases in which $C$ is Seifert fibered it will be of this form.  We illustrate  it to provide  a good model to have in mind as we develop the notation and arguments that follow.  Another good model is provided by Example~\ref{ex:sfexamp}.

\begin{figure}[h]
\labellist
\pinlabel {\text{\Large{$T_1$}}} at  25 -10
\pinlabel {\text{\Large{$f$}}} at  15  40
\pinlabel {\text{\Large{$T_1$}}} at  25 -10
\pinlabel {\text{\Large{$T_l$}}} at 65 -10
\pinlabel {\text{\Large{$T_{l+1}$}}} at 122 -10
\pinlabel {\text{\Large{$T_p$}}} at 160 -10
\pinlabel {\text{\Large{$S_i$}}} at 210 -10
\pinlabel {\text{\Large{$S_j$}}} at -10 80
\endlabellist
\includegraphics[scale=.80]{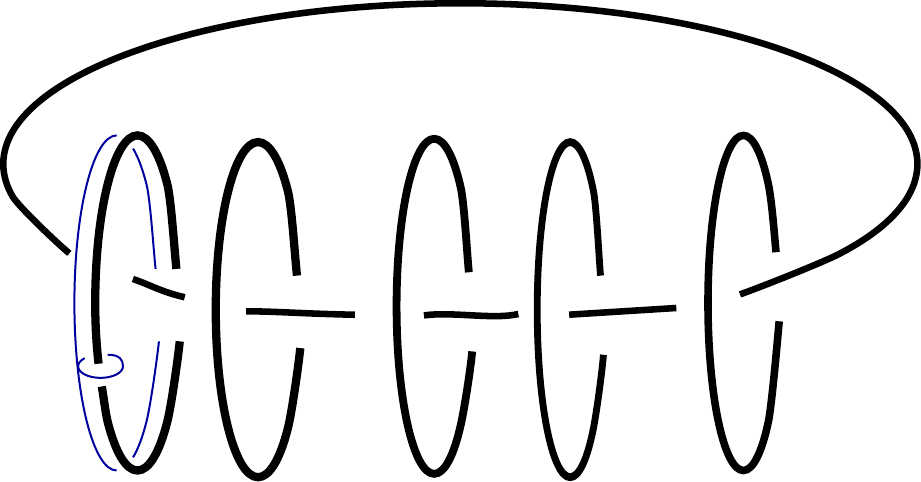}  
\vskip.15in
\caption{Possible Seifert fibered core $C$.}
\label{fig:seifertfibered}
\end{figure}
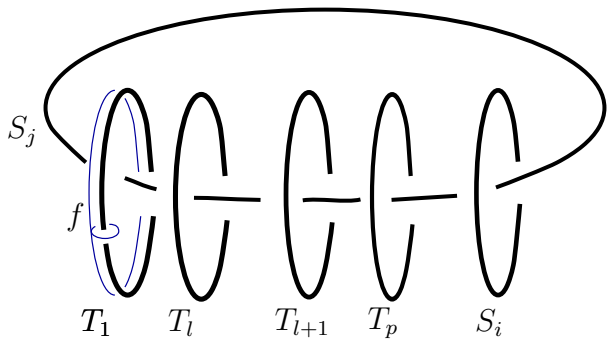

Notice the complement of this link is homeomorphic to the complement of the link formed by giving the parallel strands a full twist.  In this case, all the components, including the horizontal one, are fibers of the Hopf fibration of $S^3$.  More generally we have the following. 

\end{example}
\begin{theorem} The core $C$ is diffeomorphic to  the complement of a set of  regular fibers in a closed Seifert manifold.
\end{theorem}
\begin{proof}   Build a manifold $M$ by attaching solid tori to the boundary components of $C$ so that each longitude is identified with the fiber of the fibration of $C$.  Then the Seifert fibration of $C$ extends to $M$ and the cores of the solid tori are regular fibers.
\end{proof}

We now fix the choice of that $M$ and its Seifert fibration.

\subsection{Notation and a basis for $H_1(T_i)$}

For each $T_i$, there is a basis of $H_1(T_i)$ represented by a pair of curves, $\{f_i, g_i\}$:  since $T_i$ bounds the solid torus neighborhood of  a regular fiber, we let $f_i$ denote the fiber and let $g_i$ denote the meridian of the solid torus.

Each torus  $T_i$ bounds a  submanifold  of $S^3$ that contains the component $L_i$; denote it be  $W_i$.  All the pairs $(W_i,L_i)$ are diffeomorphic, so we choose one and denote it $(W,K)$ with boundary $T$.  We have that $T$ contains a canonical longitude  that is null-homologous in $W$ which we denote $\lambda$;  chose a second curve intersecting it once and denote it $\mu$.

We now see that $(S^3, L)$ is built from $C$ by attaching copies of  $W$ to the $T_i$ using  attaching maps we denote $G_i$.   (Other manifolds have to be attached along the other boundary components of $C$, which we have denoted $S_i$.)   Denote the images of $\{ \lambda, \mu  \}$  under $G_i$ by $\{ \lambda_i , \mu_i \} $.

\begin{theorem}  The intersection number of $\lambda_i$ with $f_i$ is nonzero.   \end{theorem}
\begin{proof}  Our proof depends on the uniqueness of the fibrations of Seifert fibered manifolds, up to isotopy.  This does not hold for all Seifert manifolds (e.g.~$S^1\times B^2$), but Waldhausen~\cite{MR235576} proved that if the Seifert fibered manifold $M$ is sufficiently large, that is, if it contains an incompressible surface that is not boundary parallel, then the fibration is unique.  (See also the reference~\cite{MR0426001}.)  In the case that the three-manifold has four or more boundary components, it is clearly sufficiently large.  The preimage of a circle in the base space that bounds two of the boundary components is an incompressible torus and is not boundary parallel. 

We now claim   that the $\lambda_i$ are not fibers of the fibration. Consider $ i \ne j$ and the pair $\lambda_i$ and $\lambda_j$.  Any element of $\diff(L)$ that maps $L_i$ to $L_j$ carries $\lambda_i$ to $\pm\lambda_j$.  Self-homeomorphisms of Seifert fibered spaces with more than three boundary components preserve fibers up to isotopy, so if $\lambda_i$ is a fiber, then $\lambda_j$ is also a fiber.  

Suppose that are $\lambda_i$ and $\lambda_j$ are fibers.  Then there is a vertical  annulus  $A$ in $C$ joining $\lambda_i$ to $\lambda_j$.  There are also surfaces $B_i$ and $B_j$ in $W_i$ and $W_j $ with boundaries $\lambda_i$ and $\lambda_j$. The union of $A$ with $B_i \cup B_j$ is a closed surface in $S^3$.   There is also a curve on $T_i$ meeting this surface in exactly one point.  This is impossible in $S^3$.

\end{proof}

\subsection{Maps between the $T_i$.}

Without loss of generality, we will focus on $T_1$ and $T_2$.  We denote a chosen element in $\diff(L)$ that carries $L_1$ to $L_2$ by $F$.  Note that we can assume $F(f_1) = f_2$, $F(\lambda_1) = \lambda_2$, and $F(\mu_1) = \mu_2$.  However, maps of $C$ do  not necessarily preserve the $g_i$.     We can assume that $F(g_1) = g_2 + wf_2$ for some $w$.

For both values of $i$ we have constants so that 
\[ \lambda_i = \alpha_i  f_i + \beta_i g_i \hskip.5in  \mu_i = \delta_i f_i + \gamma_i g_i.\]  
Applying $F$ to the set with $i=1$ and   renaming variables, we have 
\[ \lambda_1= \alpha   f_1 + \beta  g_1 \hskip.5in  \mu_1= \delta  f_1+ \gamma g_1 \hskip.3in \text{and}\]  

\[ \lambda_2= (\alpha   f_2 + \beta  g_2) + \beta w f_2 \hskip.5in  \mu_2= ( \delta  f_2+ \gamma g_2) + \gamma wf_2.\]  

\subsection{Constructing the transposition}

\begin{theorem}\label{thm:transpose}There is a diffeomorphism $G $  of $C$ that interchanges $T_1$ and $T_2$ and is the identity on all other boundary components of $C$.  The  map $G$ can be chosen so  that it preserves the $f_1$ and satisifies $G(g_1) = g_2 + wf_2$ and $G(g_2) = g_1 -wf_2$.

\end{theorem}
\begin{proof}
Using the fact that $T_i$ are boundaries of regular fibers, there is a diffeomorphism $G$ of $C$ that interchanges $T_1$ and $T_2$ that also preserves the pairs $\{f_i, g_i\}$.  This map can be assumed to be the identity on the other components.

There is a  vertical annulus in $C$ joining $f_1$ and $f_2$.  We can perform a  $w$--fold twist along this annulus.  This is the identity map on all boundary components other than $T_1$ and $T_2$.  On $T_1$ and $T_2$  it preserves the $f_1$ and $f_2$, it maps $g_1$ to $g_1 -wf_1$ and  it maps  $g_2$ to $g_2 + wf_2$.

\end{proof}

\subsection{Main theorem in Seifert fibered case}
\begin{theorem} If $\A_n \subset \S(L)$ and the associated core  $C$ is Seifert fibered, then there is an element $H \in \diff(L)$ which transposes $L_1$ and $L_2$.  Equivalently, $\S(L) = \Sn$. \end{theorem}

\begin{proof} The map $G$ given in Theorem~\ref{thm:transpose}  satisfies

\[ G(\lambda_1) = \alpha f_2 + \beta(g_2 +wf_2) \hskip.2in \text{and} \hskip.2in G(m_1) = \delta f_2 + \gamma (g_2 +wf_2) .
\]
It also satisfies
\[ G(\lambda_2) = \alpha f_1 + \beta(g_1 -wf_1) +\beta w f_1 \hskip.2in \text{and} \hskip.2in G(\mu_1) = \delta f_1 + \gamma (g_1 -wf_1)  +\gamma w f_1.
\] 
Simplifying shows that this interchanges the attaching maps of $W$ to $T_1$ and $T_2$, and thus extends as desired. 
\end{proof}

\section{Questions}\label{sec:questions}

\begin{enumerate}
\item  Figure~\ref{fig:a4} illustrates a 4--component  link  $L$ satisfying  $ \S(L) = \A_4$.   This example was found by  Nathan Dunfield using Snappy~\cite{SnapPy} .  In the Snappy enumeration of links, it is  $L12a2007$.  Does there exist a 5--component link $L$ with $\S(L) = \A_5$.  (The group $\S_4$ has 11 conjugacy classes of subgroups.  Each is $\S(L)$ for some link.  The case of $\A_4$ is the most difficult to find.)
\begin{figure}[h]
\includegraphics[scale=.15]{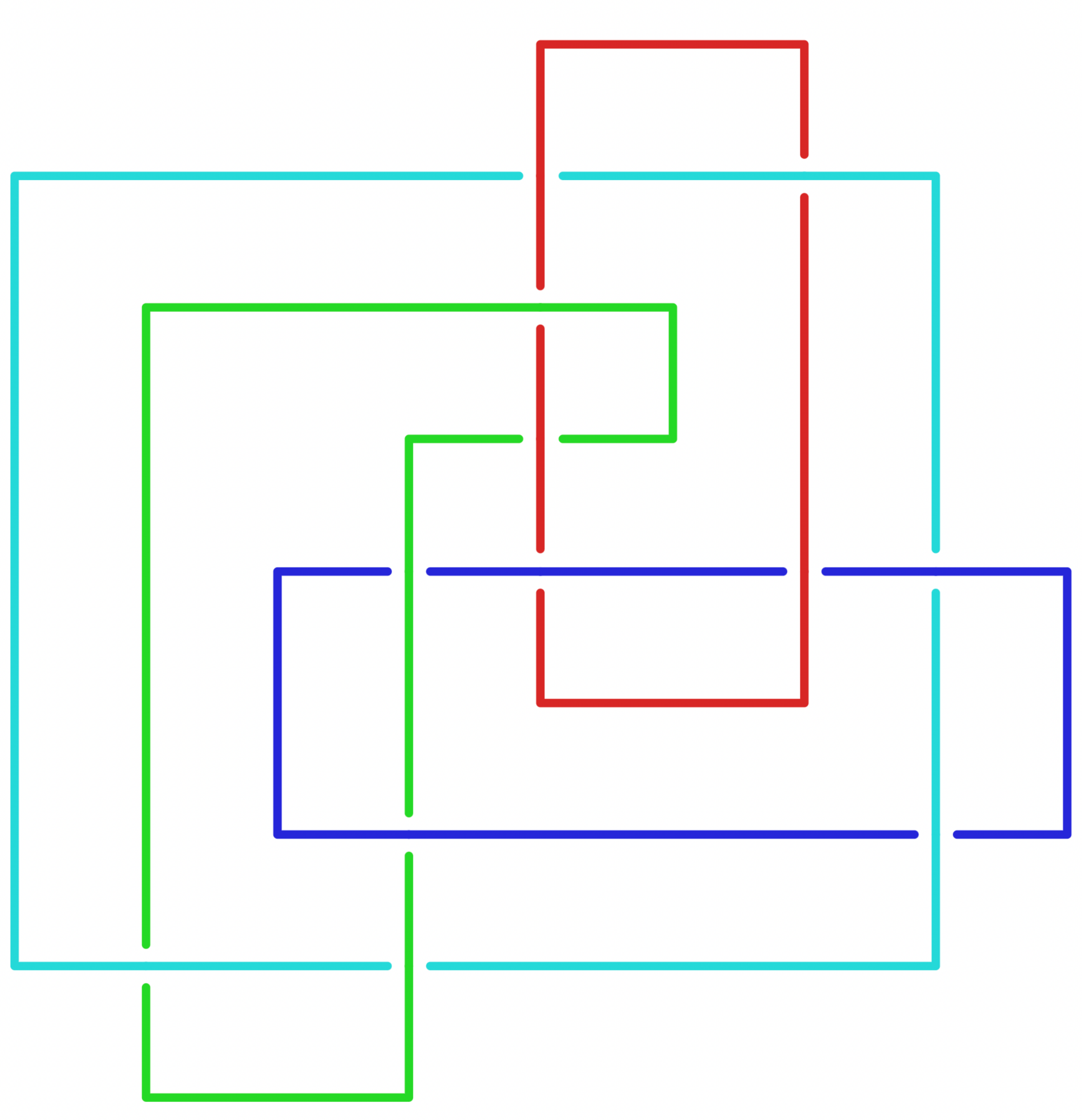}  
\caption{A link with intrinsic symmetry group $A_4$.}
\label{fig:a4}
\end{figure}

\item \label{question:fw} The Fox-Whitten group $\Gamma_n$ maps onto $\S_n$, and thus the obstructions we have developed here provide obstructions to groups $G \subset \Gamma_n$ from being oriented intrinsic symmetry groups of links.  Can the techniques used here provide finer obstructions in the oriented case?

\item  As a particular example of Quesiton~\ref{question:fw}, can  any of the unknown cases for $2$--component links described in Section~\ref{sec:links} be eliminated as possible intrinsic symmetry groups?

\item If a subgroup $H\subset \S_n$ or $H \subset \Gamma_n$ is the intrinsic symmetry group for a link, is it the intrinsic symmetry group of a non-split link or of an irreducible link?

\item \label{prob:brunnian}  A natural class of links consists of Brunnian links; these are non-splittable but become the unlink upon removing any one of the components.  The links produced in Theorem~\ref{thm:fullsym} having symmetry group $\Sn$ are not Brunnian.  The example of 2--component and 3--component links with $\S(L) = \Sn$ that proceed the proof of that theorem are Brunnian.  Hence, we ask: for all $n\ge 3$, does there exist a Brunnian link $L$ with $\S(L) = \Sn$?

\item Another class to consider is alternating links, and presumably there are strong constraints on $\S(L)$ for these.  

\item  Let $M$ be a  compact three-manifold with $n$ torus boundary components, $\partial_i(M)$.  Choose a basis of $H_1(\partial_i(M))$ for each $i$.  One can form a Whitten-like group   $\Omega_n =\Z_2 \oplus (G^n  \rtimes \S_n)$ where $G$ is the automorphism group of  $\Z \oplus \Z$.   Each  manifold $M$ gives rise to a subgroup of $\Omega_n$.   What subgroups arise in this way?  This is particularly interesting in the case that the interior of $M$ has a complete hyperbolic structure.

\item  The previous question can be modified.  Given a subgroup $H \subset \S_n$, is there a complete  hyperbolic three-manifold with $n$ cusps such the $H$ represents the permutations of the cusps that are realized by isometries of $M$?  In relation to this, Paoluzzi and Porti~\cite{MR2493374} proved that every finite group is the isometry group of the complement of a hyperbolic link in $S^3$.   Notice that their isometries need not extend to $S^3$.  Applying their construction to a subgroup of $\S_n$ does not  produce an $n$ component link.
\end{enumerate}



\bibliography{../../../BibTexComplete}

\begin{bibdiv}
\begin{biblist}

\bib{MR2909632}{article}{
      author={Berglund, Michael},
      author={Cantarella, Jason},
      author={Casey, Meredith~Perrie},
      author={Dannenberg, Eleanor},
      author={George, Whitney},
      author={Johnson, Aja},
      author={Kelley, Amelia},
      author={LaPointe, Al},
      author={Mastin, Matt},
      author={Parsley, Jason},
      author={Rooney, Jacob},
      author={Whitaker, Rachel},
       title={Intrinsic symmetry groups of links with 8 and fewer crossings},
        date={2012},
        ISSN={2073-8994},
     journal={Symmetry},
      volume={4},
      number={1},
       pages={143\ndash 207},
         url={https://doi.org/10.3390/sym4010143},
      review={\MR{2909632}},
}

\bib{MR2178962}{article}{
      author={Boileau, Michel},
      author={Leeb, Bernhard},
      author={Porti, Joan},
       title={Geometrization of 3-dimensional orbifolds},
        date={2005},
        ISSN={0003-486X},
     journal={Ann. of Math. (2)},
      volume={162},
      number={1},
       pages={195\ndash 290},
         url={https://doi.org/10.4007/annals.2005.162.195},
      review={\MR{2178962}},
}

\bib{MR2300613}{article}{
      author={Budney, Ryan},
       title={J{SJ}-decompositions of knot and link complements in {$S^3$}},
        date={2006},
        ISSN={0013-8584},
     journal={Enseign. Math. (2)},
      volume={52},
      number={3-4},
       pages={319\ndash 359},
      review={\MR{2300613}},
}

\bib{rbudney}{article}{
      author={Budney, Ryan},
       title={Math overflow: Is there a 2 component link with full symmetry?},
        date={2012},
      eprint={http://mathoverflow.net/questions/91493},
}

\bib{MR2909631}{article}{
      author={Cantarella, Jason},
      author={Cornish, James},
      author={Mastin, Matt},
      author={Parsley, Jason},
       title={The 27 possible intrinsic symmetry groups of two-component
  links},
        date={2012},
        ISSN={2073-8994},
     journal={Symmetry},
      volume={4},
      number={1},
       pages={129\ndash 142},
         url={https://doi.org/10.3390/sym4010129},
      review={\MR{2909631}},
}

\bib{SnapPy}{misc}{
      author={Culler, Marc},
      author={Dunfield, Nathan~M.},
      author={Goerner, Matthias},
      author={Weeks, Jeffrey~R.},
       title={Snap{P}y, a computer program for studying the geometry and
  topology of $3$-manifolds},
         how={Available at \url{http://snappy.computop.org} (DD/MM/YYYY)},
        date={2019},
}

\bib{MR683753}{article}{
      author={Hartley, Richard},
       title={Identifying noninvertible knots},
        date={1983},
        ISSN={0040-9383},
     journal={Topology},
      volume={22},
      number={2},
       pages={137\ndash 145},
         url={https://mathscinet.ams.org/mathscinet-getitem?mr=683753},
      review={\MR{683753}},
}

\bib{MR867798}{article}{
      author={Hillman, Jonathan~A.},
       title={Symmetries of knots and links, and invariants of abelian
  coverings. {I}},
        date={1986},
        ISSN={0289-9051},
     journal={Kobe J. Math.},
      volume={3},
      number={1},
       pages={7\ndash 27},
      review={\MR{867798}},
}

\bib{MR559040}{article}{
      author={Kawauchi, Akio},
       title={The invertibility problem on amphicheiral excellent knots},
        date={1979},
        ISSN={0386-2194},
     journal={Proc. Japan Acad. Ser. A Math. Sci.},
      volume={55},
      number={10},
       pages={399\ndash 402},
         url={https://mathscinet.ams.org/mathscinet-getitem?mr=559040},
      review={\MR{559040}},
}

\bib{MR0426001}{book}{
      author={Orlik, Peter},
       title={Seifert manifolds},
      series={Lecture Notes in Mathematics, Vol. 291},
   publisher={Springer-Verlag, Berlin-New York},
        date={1972},
      review={\MR{0426001}},
}

\bib{MR2493374}{article}{
      author={Paoluzzi, Luisa},
      author={Porti, Joan},
       title={Hyperbolic isometries versus symmetries of links},
        date={2009},
        ISSN={0166-8641},
     journal={Topology Appl.},
      volume={156},
      number={6},
       pages={1140\ndash 1147},
         url={https://doi.org/10.1016/j.topol.2008.10.008},
      review={\MR{2493374}},
}

\bib{MR2416241}{article}{
      author={Porti, Joan},
       title={Geometrization of three manifolds and {P}erelman's proof},
        date={2008},
        ISSN={1578-7303},
     journal={Rev. R. Acad. Cienc. Exactas F\'{\i}s. Nat. Ser. A Mat. RACSAM},
      volume={102},
      number={1},
       pages={101\ndash 125},
         url={https://doi.org/10.1007/BF03191814},
      review={\MR{2416241}},
}

\bib{MR72482}{article}{
      author={Schubert, Horst},
       title={Knoten und {V}ollringe},
        date={1953},
        ISSN={0001-5962},
     journal={Acta Math.},
      volume={90},
       pages={131\ndash 286},
         url={https://doi.org/10.1007/BF02392437},
      review={\MR{72482}},
}

\bib{MR0158395}{article}{
      author={Trotter, H.~F.},
       title={Non-invertible knots exist},
        date={1963},
        ISSN={0040-9383},
     journal={Topology},
      volume={2},
       pages={275\ndash 280},
         url={https://mathscinet.ams.org/mathscinet-getitem?mr=0158395},
      review={\MR{0158395}},
}

\bib{MR235576}{article}{
      author={Waldhausen, Friedhelm},
       title={Eine {K}lasse von {$3$}-dimensionalen {M}annigfaltigkeiten. {I},
  {II}},
        date={1967},
        ISSN={0020-9910},
     journal={Invent. Math.},
      volume={3},
       pages={308\ndash 333; ibid. 4 (1967), 87\ndash 117},
         url={https://doi.org/10.1007/BF01402956},
      review={\MR{235576}},
}

\bib{MR242146}{article}{
      author={Whitten, W.~C., Jr.},
       title={Symmetries of links},
        date={1969},
        ISSN={0002-9947},
     journal={Trans. Amer. Math. Soc.},
      volume={135},
       pages={213\ndash 222},
         url={https://doi.org/10.2307/1995013},
      review={\MR{242146}},
}

\end{biblist}
\end{bibdiv}
\bibliographystyle{plain}	

\end{document}